\documentclass{c0barcode}

\begin{document}

\title{The action spectrum and $C^0$ symplectic topology}
\date{\today}
\author{Lev Buhovsky, Vincent Humili\`ere, Sobhan Seyfaddini}
\maketitle

\begin{abstract} 
  Our first main result states that the spectral norm $\gamma $ on $ \Ham(M, \omega) $, introduced  in the works of Viterbo, Schwarz and Oh, is continuous with respect to the $C^0$ topology, when $M$ is  symplectically aspherical. This statement was previously proven only  in the case of closed surfaces. As a corollary, using a recent result of Kislev-Shelukhin, we obtain $C^0$ continuity of barcodes on aspherical symplectic manifolds, and furthermore define barcodes for Hamiltonian homeomorphisms.  We also present several applications to Hofer geometry and dynamics of Hamiltonian homeomorphisms.

Our second main result is related to the Arnold conjecture about fixed points of Hamiltonian diffeomorphisms. The recent example of a Hamiltonian {\it homeomorphism} on any closed symplectic manifold of dimension greater than $ 2 $ having only one fixed point,  shows that the conjecture does not admit a direct generalization to the $ C^0 $ setting. However, in this paper we demonstrate that a reformulation of the conjecture in terms of fixed points as well as spectral invariants still holds for Hamiltonian homeomorphisms  on symplectically aspherical manifolds.
 \end{abstract}

\setcounter{tocdepth}{2}
\tableofcontents

%%%%%%%%%%%%%%%%%%%%%%%%%%%%%%%%%%%%%%%%%%%%%%%%%%
%%%%%%%%%%%%%%%%%%%%%%%%%%%%%%%%%%%%%%%%%%%%%%%%%%%%%%
%%%%%%%%%%%%%%%%%%%%%%%%%%%%%%%%%%%%%%%%%%%%%%%%%%%%%%%

\section{Introduction}
 As a consequence of Floer's proof of  the Arnold conjecture, we now know that, on fairly general closed and connected symplectic manifolds, a Hamiltonian diffeomorphism must necessarily possess a large number of fixed points.  On the other hand,  Hamiltonian \emph{homeomorphisms}, which are defined as those homeomorphisms obtained as uniform limits of Hamiltonian diffeomorphisms\footnote{We point out that the definition of Hamiltonian homeomorphisms that we adopt in this paper agrees with the one of Le Calvez \cite{lecalvez06}, but differs from that of M\"uller and Oh \cite{muller-oh} which is more restrictive.}, behave quite differently:    We recently proved in \cite{BHS} that every closed and connected symplectic manifold of dimension at least four admits Hamiltonian homeomorphisms with just a single fixed point.   This surprising behaviour is a higher dimensional phenomenon as it is well known that a Hamiltonian homeomorphism has at least two fixed points on the sphere and three on surfaces of higher genus. This fact was proven by Matsumoto \cite{matsumoto}; see also \cite{Franks, lecalvez05}.

 This article addresses the question of $C^0$ continuity of certain symplectic invariants, usually referred to as spectral invariants,  which are extracted from the action spectrum of a Hamiltonian diffeomorphism via Floer theory.  Questions of this nature were first raised by Viterbo in \cite{viterbo} where he proved that these invariants are $C^0$ continuous in the case of $\R^{2n}$.  The case of closed manifolds was studied in \cite{Sey13} where it is shown that such invariants are $C^0$ continuous on surfaces.\footnote{To be more precise, it is proven in \cite{Sey13} that the spectral norm is $C^0$ continuous on surfaces.}   The arguments given in \cite{Sey13} rely heavily on two dimensional fragmentation techniques which do not generalize to higher dimensions.  In light of the aforementioned counterexample to the Arnold conjecture, one might expect the results on  $C^0$ continuity of spectral invariants not to extend beyond surfaces.  However, we will see in this article that although $C^0$ rigidity of fixed points fails in  dimensions four and above, the results of \cite{Sey13} do generalize to these dimensions.  In fact, we will show that the entire action spectrum is $C^0$ continuous in a very precise sense which will be explained below.
 
 As a consequence of the above results we will be able to define spectral invariants for arbitrary Hamiltonian  homeomorphisms.  This in turn allows us to present a generalization of the Arnold conjecture which continues to hold for Hamiltonian homeomorphisms:  We prove that, in spite of the counter-example from \cite{BHS}, the cup length estimate of the homological version of the Arnold conjecture survives in the $C^0$ setting if we include in the count the total number of spectral invariants.

\subsection{$C^0$-continuity of the action spectrum}
Let  $(M,\omega)$ be a closed and connected symplectic manifold and denote by $\Ham(M, \omega)$ and  $\overline{\Ham}(M, \omega)$ the groups of Hamiltonian diffeomorphisms and homeomorphisms of $(M, \omega)$, respectively; see Section \ref{sec:hamiltonian_homeos} for definitions.

Spectral invariants are \emph{homologically essential} values of the action functional which are defined in the spirit of \emph{min-max} critical value selectors from Lusternik-Schnirelmann theory.  In this paper, we consider the case of symplectically aspherical manifolds, \emph{i.e.}\ symplectic manifolds satisfying the condition $\omega|_{\pi_2(M)} = 0  = c_1|_{\pi_2(M)}$. It is known that spectral invariants are particularly well-behaved under this assumption.  Given a Hamiltonian $H: \S^1 \times M \rightarrow \R$  and  $a \in H_*(M) \setminus \{0\}$, the spectral invariant $c(a, H)$ is, roughly speaking, defined to be the action value at which the homology class $a$ appears in the Hamiltonian Floer homology of $H$; see Section \ref{sec:spec_inv} for a detailed definition.

Let $a, b \in H_*(M)$ be non-zero homology classes.  For any $\phi \in \Ham(M, \omega)$ we define the difference of spectral invariants $\gamma(a, b; \phi):= c(a, H) -c(b, H)$ where $H$ is any Hamiltonian the time--1 map of whose flow is $\phi$.  It is well-known that this difference of spectral invariants  does not depend on the choice of $H$ and so it is well-defined; see Section \ref{sec:spec_inv}.  In the specific case where $a = [M], b = [pt]$, the function $\gamma([M], [pt] ; \cdot): \Ham(M, \omega) \rightarrow \R$ induces a non-degenerate norm on $\Ham(M, \omega)$ which is referred to as the spectral norm and is simply  denoted by $\gamma(\cdot)$. In \cite[Question 5.13]{Oh10}, Oh asked whether $\gamma : \Ham(M, \omega) \rightarrow \R$ is continuous with respect to the $C^0$ topology.  Over the past decade, with the expansion  of $C^0$ symplectic topology, this question has received much attention (see \cite{Oh10, muller-oh, Sey12, Sey13, Sey13b, dore-hanlon})  and has only been answered in the case of surfaces in \cite{Sey13}.  Our main result settles this question for any closed, connected and symplectically aspherical manifold.

\begin{theo}\label{theo:gamma} Let $(M, \omega)$ be closed, connected, and symplectically aspherical.  For any $a, b \in H_*(M) \setminus \{0\}$, the difference of spectral invariants $\gamma(a,b; \cdot): \Ham(M, \omega) \rightarrow \R$ is continuous with respect to the $C^0$ topology on $\Ham(M, \omega)$ and extends continuously to $\overline{\Ham}(M, \omega)$.

In particular, the $\gamma$ norm is $C^0$ continuous and extends continuously to $\overline{\Ham}(M, \omega)$.
  \end{theo}

  \begin{remark}\label{remark:non-aspherical}  The question of whether the spectral norm on $\Ham(M,\omega)$ is $C^0$-continuous on general symplectic manifolds remains open. Besides the case of symplectically aspherical manifolds established here, this continuity was very recently proved on $\C\mathbb{P}^n$ by Shelukhin \cite{Shelukhin-Viterbo-conj} (see also \cite{Kawamoto}) and on negative monotone manifolds by Kawamoto \cite{Kawamoto}.

    On general (not necessarily aspherical) closed connected symplectic manifolds, the numbers $c(a,H)-c(b,H)$ may not only depend on the time one map $\phi_H^1$. We can think of $\gamma(a,b;\cdot)$ as a map on the space $\mathcal{P}\Ham(M,\omega)$ of smooth paths starting at the identity in $\Ham(M,\omega)$. Our proof of Theorem \ref{theo:gamma} then adapts easily to show the following statement:

    \medskip
    \emph{For all quantum homology classes $a,b\in QH_*(M)\setminus\{0\}$, the difference of spectral invariants $\gamma(a,b; \cdot): \mathcal{P}\Ham(M, \omega) \rightarrow \R$ is continuous with respect to the $C^0$-topology on $\mathcal{P}\Ham(M, \omega)$ and extends continuously to $\overline{\mathcal{P}Ham}(M, \omega)$.}
    
\medskip
 Here, of course, the $C^0$-topology on the space $\mathcal{P}\Ham(M,\omega)$ is induced by the distance $d((\phi_t)_{t\in[0,1]},(\psi_t)_{t\in[0,1]})=\sup_{t\in[0,1]} d_{C^0}(\phi^t, \psi^t)$. See Remark \ref{remark:proof-non-asperical} for more details.   
  \end{remark}

  As we will now explain, it is not only the (differences between) spectral invariants which are $C^0$ continuous, but in fact, using the theory of barcodes one can make sense of $C^0$ continuity of the entire action spectrum.  

\medskip

\noindent \textbf{Barcodes:}  A barcode $\mathcal{B}= \{I_j\}_{1 \leq j \leq N}$ is a finite collection of intervals (or bars) $I_j = (a_j, b_j]$, $a_j \in \mathbb{R}$, $b_j \in \mathbb{R}\cup \{ +\infty\}$.  The space of barcodes can be equipped with the so-called bottleneck distance which will be denoted by $\dB$; see e.g. \cite{Oudot-book}.  

  Using Hamiltonian Floer homology one can associate a canonical barcode $\mathcal{B}(H)$ to every Hamiltonian $ H $; see \cite{PS14, UZ}. The barcode $\mathcal{B}(H)$  encodes a significant amount of information about the Floer homology of $H$:  it completely characterizes the filtered Floer complex of $H$ up to quasi-isomorphism, and hence it subsumes all of the previously constructed filtered Floer theoretic invariants. For example, the spectral invariants of $H$ correspond to the endpoints of the half-infinite bars in $\mathcal{B}(H)$.  

  Given a barcode $\mathcal{B}= \{I_j\}_{1\leq j \leq N}$ and $c\in \mathbb{R}$ define $\mathcal{B} + c =   \{I_j + c\}_{1\leq j \leq N}$, where $I_j + c $ is the interval obtained by adding $c$ to the endpoints of $I_j$.  Let $\sim$ denote the equivalence relation on the space of barcodes given by  $\mathcal{B} \sim \mathcal{C}$  if  $\mathcal{C} = \mathcal{B} +c$ for some $c \in \mathbb{R}$; we will denote the quotient space by $\hatbarcodes$.     Now the bottleneck distance descends to a distance on $\hatbarcodes$ which we will continue to denote by $\dB$. If $H,G$ are two Hamiltonians the time--1 maps of whose flows coincide, then $\B(H) = \B(G)$ in $\hatbarcodes$ (this follows from (\ref{eq:invariance_floer}). 
  Hence, we obtain a map $\B : (\Ham(M, \omega), d_{C^0} ) \rightarrow (\hatbarcodes, \dB).$  
The question of continuity of the mapping $\B $ was first addressed by Le Roux, Viterbo, and the third author in \cite{LSV} where it is proven that $\B $ is continuous and extends to $\overline{Ham}(M, \omega)$ when $M$ is a surface.  Our next result states that the same is true for any closed and symplectically aspherical manifold.

\begin{corol}[Cor. 6 in \cite{kis-shel}]\label{corol:barcodes} Let $(M, \omega)$ be closed, connected, and symplectically aspherical.  The mapping $$\B : (\Ham(M, \omega), d_{C^0} ) \rightarrow (\hatbarcodes, \dB) $$ is continuous and extends continuously to $\overline{\Ham}(M, \omega)$.
\end{corol}
\begin{proof}
It has recently been proven by Kislev-Shelukhin \cite{kis-shel} that the following inequality holds for all Hamiltonian diffeomorphisms $\phi, \psi$:
\begin{equation}\label{eq:kis-shel}
\dB(\cal B(\phi),\cal B(\psi))\leq\tfrac12\gamma(\psi^{-1}\!\circ\phi).
\end{equation}
The result follows immediately from the above inequality and Theorem \ref{theo:gamma}.   
\end{proof}

We should point out that, in the above result, it is important to consider barcodes upto shift.  The map $\B$ is often defined using mean-normalized Hamiltonians. In this case, it takes values in the space of barcodes, as opposed to barcodes upto shift. However, this yields a discontinuous map (See Example 2.3 in \cite{Sey13}).

\subsection{The Arnold conjecture}
We will now explain how Theorem \ref{theo:gamma} allows us to present a generalization of the Arnold conjecture which continues to hold for Hamiltonian homeomorphisms.

   In the appendix to this paper, we will show, using standard arguments from dynamics, that a $C^0$ generic Hamiltonian homeomorphism has infinitely many fixed points.  Hence, our goal here will be to address the Arnold conjecture for all elements of $\overline{\Ham}(M, \omega)$ and not a generic subset of it.
   
The (homological) Arnold conjecture states that a Hamiltonian \emph{diffeomorphism} of a closed and connected symplectic manifold $(M, \omega)$ must have at least as many fixed points as the \emph{cup length} of $M$.  Cup length, denoted by $\cl(M)$, is a topological invariant of $M$ which is defined as follows:\footnote{Here, $\cap$ refers to the intersection product in homology. Cup length can be equivalently defined in terms of the cup product in cohomology.}
\begin{align*}\cl(M):= \max \{k+1\,:\, \exists \, a_1&, \ldots, a_k \in H_*(M),\,\,
 \forall i, \deg(a_i)\neq  \mathrm{dim}(M)\\ &\text{ and }  a_1 \cap \cdots \cap a_k \neq 0\}.
\end{align*}
This version\footnote{The original version of the Arnold conjecture, in which the lower bound for the number of fixed points is predicted to be the minimal number of critial points of a smooth function on $M$, has also been established on ashperical manifolds; see \cite{rudyak-oprea}.} of the Arnold conjecture was proven, for Hamiltonian diffeomorphisms, on $\bb CP^n$ \cite{fortune, FW}  and on symplectically aspherical manifolds  \cite{floer89, hofer, rudyak-oprea}.  We should emphasize that cuplength estimates have not been established for general  monotone symplectic manifolds and, in fact,  Floer himself tended to ``believe that there are more than technical reasons for this"; see Page 577 of \cite{floer89}.
  
  \medskip 
   It was proven by Matsumoto \cite{matsumoto} that Hamiltonian \emph{homeomorphisms} of surfaces satisfy the Arnold conjecture; see also \cite{Franks, lecalvez05}.  However, we showed in \cite{BHS} that every closed and connected symplectic manifold of dimension at least 4 admits a Hamiltonian homeomorphism with a single fixed point.  
  
  This is where Theorem \ref{theo:gamma} enters the scene: the result allows us to define the action spectrum of a Hamiltonian homeomorphism (upto a shift).  In particular, we can now make sense of \emph{the total number of spectral invariants of a Hamiltonian homeomorphism.}  The theorem below shows that, in spite of the counter-example from \cite{BHS}, the cup length estimate from the homological Arnold conjecture survives if we include in the count the total number of spectral invariants.
  
  We need the following notion before stating the result:   A subset $A\subset M$ is homologically non-trivial if for every open neighborhood $U$ of $A$ the map $i_*: H_j(U) \rightarrow H_j(M)$, induced by the inclusion $i: U \hookrightarrow M$, is non-trivial for some $j>0$. Clearly, homologically non-trivial sets are infinite.

\begin{theo}\label{theo:main}
Let $(M, \omega)$ denote a closed, connected and symplectically aspherical manifold.  
Let $\phi\in \overline{\Ham}(M, \omega)$ be a Hamiltonian homeomorphism.  If the total number of spectral invariants of $\phi$ is smaller than $\mathrm{cl}(M)$, then the set of fixed points of $\phi$ is homologically non-trivial, hence is infinite.  
\end{theo}

In the smooth case, Theorem \ref{theo:main} was established by Howard \cite{Howard2012}, and our proof is inspired by his. For a smooth Hamiltonian diffeomorphism, spectral invariants correspond to actions of certain fixed points. Therefore, Theorem \ref{theo:main} is a generalization of the Arnold conjecture in the smooth setting.  However, when it comes to Hamiltonian homeomorphisms, there is a total breakdown  in the correspondence between spectral invariants and actions of fixed points:  Indeed, the Hamiltonian homeomorphism we construct in \cite{BHS} has a single fixed point and many\footnote{The set of spectral invariants of this Hamiltonian homeomorphism coincide with the spectral invariants of a $C^2$--small Morse function.  Hence, their count is at least the cup length of the manifold.  This is perhaps an indication that, on symplectic manifolds of dimension at least four, one cannot define the notion of  action for fixed points of an arbitrary Hamiltonian homeomorphism. See Remark 20 in \cite{BHS}.} distinct spectral invariants.

\subsection{Further consequences of continuity of $\gamma$} \label{sec:consequences}

One of the fascinating aspects of symplectic topology is the existence of an intriguing interplay between flexible (soft) and rigid (hard) sides of the subject.  This interplay permeates through $C^0$ symplectic topology as well:  Coisotropic submanifolds (and even their reductions) as well as symplectic submanifolds of co-dimension $ 2 $ are $C^0$ rigid \cite{HLS13, Opshtein, HLS14, BuOp}, but subcritical isotropic submanifolds,  symplectic submanifolds of codimension greater than two, and even  the most basic notion in symplectic geometry, that of symplectic area, are $C^0$ flexible \cite{BuOp}. 
 
 Theorem \ref{theo:gamma} and Corollary \ref{corol:barcodes}, together with
the recent $C^0$ counterexample to the Arnold conjecture \cite{BHS},  point towards a miraculous tale  of flexibility and rigidity:   In dimensions greater than two, fixed points are flexible, but the action spectrum and its barcode structure are rigid! 

Theorem \ref{theo:gamma} and Corollary \ref{corol:barcodes} yield new applications which manifest rigidity on higher dimensional symplectic manifolds, some of which are listed below. These applications were known to hold in dimension $2$.  Whether they would extend to higher dimensions was a mystery given the aforementioned flexibility results in higher dimensions. The applications can be derived using the same arguments as in dimension 2, from the $C^0$-continuity of the spectral norm and of barcodes.  Therefore, we will not provide any proofs but only refer to the relevant papers treating the analogous two-dimensional results.

\medskip
\subsubsection{The displaced disks problem.}
The displaced disks\footnote{The original question was posed in the two-dimensional setting, whence the use of the word ``disk''.} problem, posed by F. B\'eguin, S. Crovisier, and F. Le Roux, asks if a $C^0$ small Hamiltonian homeomorphism can displace a \emph{large} symplectic ball.  We will show that the answer is negative on all symplectically aspherical manifolds.   The case of closed surfaces was resolved in \cite{Sey13b}.

By a symplectic ball we mean the image of a symplectic embedding $i: (B, \omega_0) \rightarrow (M, \omega)$, where $(B, \omega_0)$ denotes a closed Euclidean ball  equipped with the standard symplectic structure.  If we know that $B$ has radius $r$, we then refer to its image as a symplectic ball of radius $r$.

\begin{theo}\label{theo:displaced_disks}
Let $(M, \omega)$ be closed, connected and symplectically aspherical.   For every $r>0$, there exists $\epsilon >0$ with the following property: if $\phi \in \overline{\Ham}(M, \omega)$ displaces a symplectically embedded ball of radius $r$, then $d_{C^0}(\id, \phi) > \epsilon$.  
\end{theo}

The above result tells us that Hamiltonian homeomorphisms which are small in the $C^0$ sense cannot displace large sets.  This may be interpreted as a $C^0$ analogue of the celebrated energy-capacity inequality \cite{hofer90, lalonde-mcduff, usher10}.

\subsubsection{Rokhlin groups and almost  conjugacy.}
We will be addressing the following question  of B\'eguin, Crovisier, and Le Roux: Does $\overline{\Ham}(M, \omega)$ possess a dense conjugacy class?  The fact that the answer to this question is negative is a consequence of Theorem \ref{theo:displaced_disks}. The case of surfaces was resolved in \cite{EPP, Sey13b}. 
The question of existence of topological groups which possess dense conjugacy classes is of interest in ergodic theory; see \cite{glasner-weiss01, glasner-weiss08}.  Glasner and Weiss refer to such groups as \emph{Rokhlin} groups.   An interesting example of  a Rokhlin group is  the identity component of 
the group of  homeomorphisms of  any even dimensional sphere equipped with the topology of uniform convergence. For further examples see \cite{glasner-weiss01, glasner-weiss08}.

Studying the above question naturally leads to the consideration of  an equivalence relation called \emph{almost conjugacy}:   This is the smallest Hausdorff equivalence relation which is larger than the conjugacy relation\footnote{The almost conjugacy relation may be characterized by the following universal property: $\varphi \sim \psi$ if and only if $f(\varphi) = f(\psi)$ for any continuous function $f: \overline{\Ham}(M, \omega) \to Y$ such that $f$ is invariant under conjugation and $Y$ is a Hausdorff topological space.};  see \cite{LSV} for further details. 
It is introduced and studied  extensively, in the context of closed surfaces, in \cite{LSV}.  An important feature of almost conjugacy is that in Rokhlin groups any two elements are almost conjugate, and hence the relation is trivial for such groups.\footnote{A notion very closely to that of almost conjugacy, called $\chi$--equivalence, arises naturally in the study of surface group actions on the circle; see \cite{MannWolff} and references therein.}

The two theorems below were first proven in the two-dimensional setting in \cite{LSV}.  Here, we extend them to higher dimensional  symplectically aspherical manifolds.

\medskip

In the theorem below, $\Fix_c(\varphi)$ denotes the set of contractible fixed points of a Hamiltonian diffeomorphism $\varphi$.  Given an isolated point $x \in \Fix_c(\varphi)$, we denote by $r(\varphi,x)$ the rank of the local Floer homology groups of $\varphi$ at the point $x$.  We remark that if $x$ is a non-degenerate fixed point of $\varphi$, then $r(\varphi,x) =1$.
 
 \begin{theo}\label{theo:almost_conjugacy1}
 Let $(M, \omega)$ be closed, connected and symplectically aspherical. 
 Let $\varphi, \psi$ be two Hamiltonian diffeomorphisms with finitely many contractible fixed points.  If $\varphi$ is almost conjugate to $\psi$ in $ \overline{\Ham}(M,\omega) $, then 
 
  $$\sum_{x\in \Fix_c(\varphi)} r(\varphi,x) = \sum_{x\in \Fix_c(\psi)}r(\psi,x) .$$ 
  
  In particular, if $\varphi,\psi$ are non-degenerate, then they have the same number of fixed points.
 \end{theo}
  Observe that as a direct consequence of the above theorem we see that the almost conjugacy relation on $\overline{\Ham}(M, \omega)$ is non-trivial and so  $\overline{\Ham}(M, \omega)$ is not a Rokhlin group. 

\medskip
 Our second result on the almost conjugacy relation tells us that barcodes are capable of detecting the wild dynamics of  homeomorphisms.  
 \begin{theo}\label{theo:almost_conjugacy2}
 Let $(M, \omega)$ be closed, connected and symplectically aspherical.  
 There exists a Hamiltonian homeomorphism $\varphi$ which is not almost conjugate to any Hamiltonian diffeomorphism.  In particular, the closure of the conjugacy class of $\varphi$ contains no Hamiltonian diffeomorphisms.
 \end{theo}
 
 It would be interesting to know whether analogues of the above result hold for other (not necessarily symplectic) transformation groups.  Of course, one would first have to know that the transformation group in question is not Rokhlin.

\subsubsection{An application to Hofer geometry}  
We will answer the following question of Le Roux \cite{LeRoux-6Questions} on certain classes of symplectic manifolds:    For any $A > 0$, let $E_A$ be the complement of the closed ball of radius $ A $ in Hofer's metric, i.e. 
$E_A:= \{ \phi \in \Ham(M, \omega):  d_H(\id, \phi) > A\}$.  Does $E_A$ have non-empty $C^0$ interior for all $A >0$? This question has been answered affirmatively in certain settings; see \cite{EPP, Sey12}.

\begin{theo}\label{theo:question_leroux}
Let $(M, \omega)$ be closed, connected and symplectically aspherical.  If the $\gamma$ norm is unbounded on $(M, \omega)$, then the set $E_A$ has non-empty $C^0$-interior.
\end{theo}

It is expected, but not proven, that $\gamma$ is unbounded on all symplectically aspherical manifolds. It is known that, $\gamma$ is unbounded  on products  of the form $(\Sigma, \omega_1) \times (N, \omega_2)$,  where $\Sigma$ is a closed surface other than the sphere. However, as pointed out in \cite{EPP}, for these manifolds one can prove the above theorem by applying the energy-capacity inequality on the universal cover.

We should point out that it is expected, and can be confirmed on a large class of symplectic manifolds, that the $C^0$ interior of a ball of finite radius, in Hofer's metric, is empty: there exist Hamiltonian diffeomorphisms which are arbitrarily $C^0$ small and Hofer large.

\subsection*{Organization of the paper}  

Sections \ref{sec:hamiltonian_homeos} and \ref{sec:prel-hamilt-floer} are devoted to preliminaries on symplectic and Hamiltonian diffeomorphisms, Hofer's distance, Floer theory and spectral invariants.
In Section \ref{sec:proofs_continuity_action}, we prove the continuity of the spectral norm and Theorem \ref{theo:gamma}. Section \ref{sec:Arnold_conj_proof} contains the proof of the generalized Arnold conjecture, Theorem \ref{theo:main}. Finally, in appendix, we prove that a $C^0$-generic Hamiltonian homeomorphism admits infinitely many fixed points.

\subsection*{Aknowledgments}
  The generalization of the Arnold conjecture presented in this paper came to life after our realization that the methods used by Wyatt Howard in \cite{Howard2012} could be adapted to non-smooth settings.  We thank him for helpful communications.

 We would like to thank Pierre-Antoine Guih\'eneuf,  Helmut Hofer,   R\'emi Leclercq, Fr\'ed\'eric Le Roux, Nicolas Vichery and Claude Viterbo for helpful conversations.  Lastly, we thank Egor Shelukhin for bringing  the following to our attention: (a variant of) Inequality \eqref{eq:kis-shel} and the fact that it implies $C^0$ continuity of barcodes; these results appear in \cite{kis-shel}.
 
SS This paper was partially written during my stay at the Institute for Advanced Study in the academic year 2016-2017.  I greatly benefited from the lively research atmosphere of the IAS and would like to thank the members of the School of Mathematics for their warm hospitality.  

VH is partially supported by the ANR project ``Microlocal'' ANR-15-CE40-0007.

LB is partially supported by ISF grants 1380/13 and 2026/17, by the ERC Starting grant 757585, and by the Alon fellowship.

\section{Preliminaries from symplectic geometry}\label{sec:hamiltonian_homeos}
  For the remainder of this section, $(M, \omega)$ will denote a closed and connected symplectic manifold. Recall that a symplectic diffeomorphism is a diffeomorphism $\theta: M \to M$ such that $\theta^* \omega = \omega$.   The set of all symplectic diffeomorphisms of $M$ is denoted by $\Symp(M, \omega)$.  
  Hamiltonian diffeomorphisms constitute an important class of examples of symplectic diffeomorphisms.  These are defined as follows: A smooth Hamiltonian $H \in C^{\infty} ([0,1] \times M)$  gives rise to a time-dependent vector field $X_H$ which is defined via the equation: $\omega(X_H(t), \cdot) = -dH_t$.  The Hamiltonian flow of $H$, denoted by  $\phi^t_H$, is by definition the flow of $X_H$.  A Hamiltonian diffeomorphism is a diffeomorphism which arises as the time-one map of a Hamiltonian flow.  The set of all Hamiltonian diffeomorphisms is denoted by $\Ham(M, \omega)$; this forms a normal subgroup of $\Symp(M, \omega)$.
\subsection{Symplectic \& Hamiltonian homeomorphisms}
 We equip $M$ with a Riemannian distance $d$. Given two maps $\phi, \psi :M \to M,$ we denote
$$d_{C^0}(\phi,\psi)= \max_{x\in M}d(\phi(x),\psi(x)).$$
We will say that a sequence of maps $\phi_i : M \rightarrow M$, converges uniformly, or $C^0$--converges, to $\phi$, if $d_{C^0}(\phi_i, \phi) \to 0$ as $ i \to \infty$. Of course, the notion of $C^0$--convergence does not depend on the choice of the Riemannian metric.

\begin{definition} \label{def:sympeo}
	 A homeomorphism  $\theta : M \to M$  is said to be symplectic if it is the $C^0$--limit of a sequence of symplectic diffeomorphisms.  We will denote the set of all symplectic homeomorphisms by $\Sympeo(M, \omega)$.  
\end{definition}

The Eliashberg--Gromov theorem states that a symplectic homeomorphism which is smooth is itself a symplectic diffeomorphism. We remark that if $\theta$ is a symplectic homeomorphism, then so is $\theta^{-1}$.  In fact, it is easy to see that $\Sympeo(M, \omega)$ forms a group. 

\begin{definition}  \label{def:hameo} 
A symplectic homeomorphism $\phi $ is said to be a Hamiltonian homeomorphism if it is the $C^0$--limit of a sequence of Hamiltonian diffeomorphisms.  We will denote the set of all Hamiltonian homeomorphisms by $\overline{\Ham}(M, \omega)$. 
\end{definition}

  It is not difficult to see that  $\overline{\Ham}(M, \omega)$ forms a normal subgroup of  $\Sympeo(M, \omega)$.   It is a long standing open question whether a smooth Hamiltonian homeomorphism, which is isotopic to identity in $\Symp(M,\omega) $, is a Hamiltonian diffeomorphism or not;  this is often referred to as the $C^0$ Flux conjecture; see \cite{LMP, Sey13c, buhovsky14}.  
  
  We should add that alternative definitions for Hamiltonian homeomorphisms do exist within the literature of $C^0$ symplectic topology.  Most notable of these is a definition given by M\"uller and Oh in \cite{muller-oh} which has received much attention. A homeomorphism which is Hamiltonian in the sense of \cite{muller-oh} is necessarily Hamiltonian in the sense of Definition \ref{def:hameo} and thus, the results of this article apply to the homeomorphisms of \cite{muller-oh} as well. 
  
  \subsection{Hofer's distance}\label{sec:hofer_distance}  
  We will denote the Hofer norm on $C^{\infty}([0,1] \times M)$ by  \[ \| H \| = \int_0^1 \left( \max_{x \in M} H(t,\cdot) - \min_{x \in M} H(t, \cdot)\right) dt.\]   The Hofer distance on $\Ham(M, \omega)$ is defined via $$d_{\mathrm{Hofer}}(\phi, \psi)= \inf \Vert H-G\Vert,$$ where the infimum is taken over all $H, G$ such that $\phi^1_H = \phi$ and $\phi^1_G = \psi$.   This defines a bi-invariant distance on $\Ham(M, \omega)$.  
  
  Given $B\subset M$, we define its \emph{displacement energy} to be $$e(B):= \inf \{ d_{\mathrm{Hofer}}(\phi, \id): \phi(B) \cap B = \emptyset\}.$$
Non-degeneracy of the Hofer distance is a consequence of the fact that $e(B) >0$ when $B$ is an open set.  This was proven in \cite{hofer90, polterovich93, lalonde-mcduff}.

\section{Preliminaries on Hamiltonian Floer theory, spectral invariants}\label{sec:prel-hamilt-floer}
 Throughout this section, $(M, \omega)$ will denote a closed, connected and symplectically  aspherical manifold of dimension $2n$.  We fix a ground field $\F$, e.g.  $\Z_2, \mathbb{Q}$, or $\mathbb{C}$. Singular homology, Floer homology and all notions relying on these theories depend on the field $\F$. 

\medskip

\noindent \textbf{The action functional and its spectrum.}
 Let $\Omega(M)$ denote the space of smooth contractible loops in $ M $, viewed as maps $\R / \Z \rightarrow M$. Let $H: [0,1] \times M$ denote a smooth Hamiltonian.  The associated  action functional $\mathcal{A}_H: \Omega(M) \rightarrow \mathbb{R}$ is defined by
   $$\mathcal{A}_H(z) :=  \int_{0}^{1} H(t,z(t))dt \text{ }- \int_{D^2} u^*\omega,$$
where $u: D^2 \rightarrow M$ is a capping disk for $z$.  Note that because $\omega|_{\pi_2(M)} =0$, the value of  $\mathcal{A}_H(z)$ does not depend on the choice of $u$.

It is a well-known fact that the set of critical points of $\mathcal{A}_H$, denoted by $\Crit(\mathcal{A}_H)$, consists of  $1$--periodic orbits of the Hamiltonian flow $\phi^t_H$.   The action spectrum of $H$, denoted by $\Spec(H)$, is the set of critical values of $\mathcal{A}_H$.  The set $\Spec(H)$ has Lebesgue measure zero.  

Suppose that $H$ and $G$ are two Hamiltonians such that $\phi^1_H = \phi^1_G$.  Then, there exists a constant $C \in \R$ such that 
\begin{equation}\label{eq:invariance_spectrum}
\Spec(H) =  \Spec(G) + C,
\end{equation}
where $\Spec(G) + C$ is the set obtained from $\Spec(G)$ by adding $C$ to each of its elements.  It follows that, given a Hamiltonian diffeomorphism $\phi$, its spectrum $\Spec(\phi)$ is a subset of $\R$ which is well-defined upto a shift.

\medskip

\noindent \textbf{Hamiltonian Floer theory.}
We say that a Hamiltonian $H$ is non-degenera\-te if the graph of $\phi^1_H$ intersects the diagonal in $M \times M$ transversally.  The Floer chain complex of (non-degenerate) $H$, $CF_*(H)$, is the vector space spanned by $\Crit(\mathcal{A}_H)$ over the ground field $\F$.   The boundary map of $CF_*(H)$ counts certain solutions of a perturbed Cauchy-Riemann equation  for a chosen $ \omega $-compatible almost complex structure $ J $ on $ TM $, which can be viewed as isolated negative gradient flow lines of $\mathcal{A}_H$. There exists a canonical isomorphism, $\Phi: H_*(M) \rightarrow HF_*(H)$, between the homology of Floer's chain complex and the singular homology of $M$;  \cite{floer89,PSS}. We will denote this isomorphism by $\Phi: H_*(M) \rightarrow HF_*(H)$.

For any $a \in \R$, we will define $CF_*^a(H):= \{\sum a_z z \in CF_*(H): \mathcal{A}_H(z) < a \}.$
It turns out that the Floer boundary map preserves $CF_*^a(H)$ and hence one can define its homology $HF_*^a(H)$.  The homology groups     $HF_*^a(H)$ are referred to as the filtered Floer homology groups of $H$.

More generally, filtered Floer homology groups may be defined for any  interval of the form $(a,b) \subset \R$:  $HF_*^{(a,b)}(H)$ is defined to be the homology of the quotient complex $CF_*^{(a,b)}(H) = CF_*^b(H)/CF_*^a(H)$.  Let us remark that the filtered Floer homology groups do not depend on the choice of the almost complex structure $J$.  

We should also add that one can define filtered Floer homology even when $H$ is degenerate.  Consider $(a,b) \subset \R$ such that $-\infty \leq a,b \leq \ \infty$ are not in  $\Spec(H)$ and define $HF_*^{(a,b)}(H)$  to be $HF_*^{(a,b)}(\tilde H)$, where $\tilde H$ is non-degenerate and sufficiently $C^2$--close to $H$.  It can be shown that $HF_*^{(a,b)}(H)$ does not depend on the choice of $\tilde H$.

It turns out that filtered Floer homology groups are  in fact  invariants of the time-1 map $\phi^1_H$ in the following sense:  Suppose that  $H$ and $G$ are two Hamiltonians such that $\phi^1_H = \phi^1_G$.  Then, there exists a constant $C \in \R$ such that we have a canonical isomorphism 
\begin{equation}\label{eq:invariance_floer}
HF_*^a(H) \cong HF_*^{a+C}(G) ,  \;\;\; \forall a \in \R.
\end{equation}
As explained in Remark 2.10 of \cite{PS14}, the above is a consequence of results from \cite{Seidel, schwarz}.

\medskip
\noindent\textbf{Spectral invariants}\label{sec:spec_inv}
Spectral invariants were first introduced by Viterbo in \cite{viterbo} in the case of $\R^{2n}$.  Here, we will be closely following Schwarz \cite{schwarz} which treats the case of closed and symplectically aspherical manifolds.\footnote{See \cite{Oh05} for the construction of these invariants on general symplectic manifolds.} 

Denote by $i_a^* : HF_*^a(H) \rightarrow HF_*(H)$  the map induced by the inclusion $i_a : CF_*^a(H) \rightarrow CF_*(H)$
and let $\alpha$ be  a non-zero  homology class.  The spectral invariant $c(\alpha, H)$ is defined by
 $$c(\alpha, H) := \inf \{a \in \R: \Phi(\alpha) \in \mathrm{Im} (i_a^*) \},  $$
where  $\mathrm{Im} (i_a^*)$ denotes the image of $i_a^* : HF_*^a(H) \rightarrow HF_*(H)$. (Recall that $ \Phi $ is the canonical isomorphism between $ H_*(M) $ and $ HF_*(H) $).

It is well-known that (see \cite{schwarz}) that $|c(\alpha,H) - c(\alpha, G)| \leq \| H -G \|$, where $ \| H \| = \int_0^1 \left( \max_{x \in M} H(t,\cdot) - \min_{x \in M} H(t, \cdot)\right) dt$ denotes the Hofer norm of $H$. This allows us to define $c(\alpha,H)$ for any smooth (or even continuous) Hamiltonian: we set $c(\alpha, H):= \lim c(\alpha, H_i)$, where $H_i$ is a sequence of smooth, non-degenerate Hamiltonians such that  $ \| H -H_i \| \to 0$.

 Given two Hamiltonians $H, G$, we will denote $\bar{H}(t,x) = -H(t, \phi^t_H(x))$ and  $ H \# G (t,x) = H(t,x) + G(t, (\phi^t_H)^{-1}(x))$.  The flows of  these Hamiltonians are $(\phi^t_H)^{-1}$ and  $\phi^t_H\circ\phi^t_G$, respectively.
Spectral invariants satisfy the following properties whose proofs can be found in \cite{schwarz} as well as \cite{Oh05, Oh06, usher10}.
 \begin{prop} \label{prop:sepc_inv}
  The function $c: (H_*(M) \setminus \{0\}) \times C^{\infty}([0,1] \times M)  \rightarrow \mathbb{R}$ has the following properties:
    \begin{enumerate}
    \item  $c(\alpha,H) \in \Spec(H) $,
    \item  $c(\alpha \cap \beta ,H\#G) \leq c(\alpha,H) + c(\beta,G)$, 
    \item  $|c(\alpha,H) - c(\alpha,G)| \leq \Vert H - G \Vert$, 
    \item $c([M], H) = - c([pt], \bar{H})$, 
    \item Let $f \in C^{\infty}(M)$ denote an autonomous Hamiltonian and suppose that $\alpha \in H_*(M)$ is a non-zero homology class.  Then, for $\eps>0$ sufficiently small, 
      \[c(\alpha, \eps f) = c_{LS}(\alpha,\eps f) =  \eps \, c_{LS}(\alpha,f),\]
    where $ c_{LS}(\alpha,f)$ is the topological quantity\footnote{Here, the subscript ``LS'' refers to ``Lusternik-Schnirelman'' since this quantity is related to the so-called Lusternick-Schnirelman theory; this is discussed in further details in Section \ref{sec:LS_theory}.} defined by 
      \[ c_{LS}(\alpha,f)=\, \inf \left\{a \in \R: \alpha \in \mathrm{Im}\,(\, H_*( \{f\! <\! a \}) \rightarrow H_*(M) \right)\,\}.\]
    \end{enumerate}   
    \end {prop}
   
   \medskip 
    As a consequence of Equation \eqref{eq:invariance_floer}, spectral invariants are invariants of the time-1 map $\phi^1_H$ in the following sense:  If  $H$ and $G$ are two Hamiltonians such that $\phi^1_H = \phi^1_G$, then there exists a constant $C \in \R$ such that
\begin{equation}\label{eq:invariance_spec}
c(\alpha, H) - c(\alpha, G) = C,  \;\;\; \forall \alpha \in H_*(M) \setminus \{0\}.
\end{equation}
Hence, we see that the difference of two spectral invariants defined via  
$$\gamma(\alpha, \beta; \phi^1_H):= c(\alpha, H) - c(\beta,H)$$ depends only on $\phi^1_H$.  As mentioned in the introduction, the so-called spectral norm $\gamma : \Ham(M, \omega) \rightarrow \R$ is defined via $\gamma(\cdot) := \gamma([M], [pt]; \cdot)$.  The $\gamma$ norm satisfies the following list of properties:

\begin{enumerate}
\item Non-degeneracy: $\gamma(\phi) \geq 0$ with equality if and only $\phi =\id$,
\item Hofer boundedness: $\gamma(\phi) \leq d_{\mathrm{Hofer}}(\phi, \id)$,
\item Conjugacy invariance: $\gamma(\psi \phi \psi^{-1}) = \gamma(\phi)$ for any $\phi \in \Ham(M,\omega)$ and any $\psi \in \Symp(M, \omega)$,
\item Triangle inequality: $\gamma(\phi \psi) \leq \gamma(\phi) + \gamma(\psi)$,
\item Duality: $\gamma(\phi^{-1})=\gamma(\phi)$,
\item Energy-Capacity inequality: $\gamma(\phi) \leq 2 e(\rm{Supp}(\phi))$, where $e(\rm{Supp}(\phi))$ denotes the displacement energy of the support of $\phi$.
\end{enumerate}

Let us point out that the non-degeneracy property is an immediate consequence of the following: Let $B \subset M$ denote a symplectically embedded ball of radius $r$. If $\phi(B) \cap B = \emptyset$, then,
\begin{equation}\label{eq:energy-capacity}
\pi r^2 \leq \gamma(\phi).
\end{equation} 
The above inequality, which is also referred to as the energy-capacity inequality, follows from the results in \cite{usher10}.

\section{Proof of Theorem \ref{theo:gamma}: $C^0$-continuity of $\gamma$}
\label{sec:proofs_continuity_action}
We will start our proof by establishing the $C^0$-continuity of $\gamma$.

It was proved in \cite{Sey12} that the spectral norm $\gamma$ is $C^0$-continuous on the subset of diffeomorphisms of $\Ham(M,\omega)$ generated by Hamiltonians supported in the complement of a given open subset. Our proof of the $C^0$-continuity of $\gamma$ will consist in reducing to this case. More precisely, we will need the following slight variant of (the symplectically aspherical case of) Theorem 1 in \cite{Sey12}.

\begin{lemma}\label{lemma:variant-sword} Let $(M,\omega)$ be a closed symplectically aspherical  manifold, and let $U$ be a connected  open subset in $M$. Then, for every $\eps>0$, there exists $\Delta>0$ such that for any $\phi\in\Ham(M,\omega)$ satisfying $\phi(x)=x$ for all $x\in U$, and $d_{C^0}(\phi,\id)< \Delta$, we have $\gamma(\phi)< \eps$.
\end{lemma}

The proof is very similar to the one provided in \cite{Sey12}, with a small modification due to the fact that our map $\phi$ is not supposed to be generated by a Hamiltonian supported in $M\setminus U$.

\begin{proof} By assumption, the points of $U$ are all fixed points of $\phi$. The value of their action depends on the choice of the Hamiltonian which generates $\phi$. However, since $U$ is assumed to be connected, this value is constant on $U$, and we will denote it by $A$.

  Let $F$ be a Morse function on $M$ all of whose critical points are located in $U$. We assume that $F$ is so small that its Hamiltonian flow does not admit any other periodic orbits of length $\leq 1$ than its critical points, and that $\max F-\min F<\eps$. Thus, the spectrum of $F$ is the set of critical values of $F$.  This also implies that $\phi_F^1$ has no fixed points in $M\setminus U$. Thus, there exists $\Delta>0$ such that for all $x\in M\setminus U$, we have $d(\phi_F^1(x),x)>\Delta$ (in the terminology of \cite{Sey13, Sey12}, the map $\phi_F^1$ ``$\Delta$-shifts'' $M\setminus U$).

  As a consequence, if $d_{C^0}(\phi,\id)<\Delta$, then $\phi_F^1\circ\phi$ does not have any fixed point in $M\setminus U$. Since $\phi$ acts as the identity on $U$, we get that $\phi_F^1\circ\phi$ has the same set of fixed points as $\phi_F^1$, which is in turn the set of critical points of $F$. Moreover, the action of point $x$ is $F(x)$ if we think of $x$ as a fixed point of $\phi_F^1$, and is $A+F(x)$ if we see it as fixed point of $\phi_F^1\circ\phi$.

  Therefore, each spectral invariant $c(\alpha,\phi)$ of $\phi$ takes the form $A+F(x)$ for some critical point $x$ of $F$. In particular,
  $$\gamma(\phi_F^1\circ\phi)\leq A+\max F-A-\min F<\tfrac\eps 2.$$
  Using the triangle inequality, we deduce that under the condition $d_{C^0}(\phi,\id)<\Delta$, we have: $$\gamma(\phi)\leq\gamma(\phi_F^{-1}) + \gamma(\phi_F^1\circ\phi) < \tfrac \eps 2 + \tfrac \eps 2=\eps.$$
\end{proof}

In order to reduce to Lemma \ref{lemma:variant-sword}, we will use a trick which consists in doubling coordinates by introducing the auxiliary map:
\begin{align*}\Phi=\phi\times\phi^{-1}:\ M\times M &\to M\times M,\\
   (x,y)&\mapsto (\phi(x),\phi^{-1}(y)),
\end{align*}
where we endow $M\times M$ with the symplectic form $\omega\oplus\omega$.
The map $\Phi$ is a Hamiltonian diffeomorphism. More precisely, if $\phi$ is the time-1 map of a Hamiltonian $H$, then $\Phi$ is the time-1 map of the Hamiltonian $K: [0,1]\times M\times M\to \R$, $(t,x,y)\mapsto H(t,x)-H(t,\phi_H^t(y))$. Moreover, if $\phi$ is $C^0$ close to the identity, so is $\Phi$. According to the product formula for spectral invariants (Theorem 5.1 in \cite{Entov2009}), we have $c(K)=c(H)+c(\bar H)=\gamma(\phi)$ and similarly,  $c(\bar{K})=c(\bar{H})+c(H)=\gamma(\phi)$. Thus,
\begin{equation}\label{eq:Phi-phi}\gamma(\Phi)=2\gamma(\phi).
\end{equation}
We will prove the following Lemma.

\begin{lemma}\label{lemma:trick} For any ball $B$ in $ M $ there exists a smaller ball $ B' \subset B $ with the following property. For any $ \Delta > 0 $ there exists $ \delta > 0 $ such that if $\phi\in \Ham(M,\omega)$ satisfies $d_{C^0}(\phi,\id_M)<\delta$, then one can find a Hamiltonian diffeomorphism $\Psi\in \Ham(M\times M,\omega\oplus\omega)$ satisfying the following properties:
  \begin{enumerate}[(i)]
  \item $ \supp(\Psi) \subset B \times B$ and $\supp(\Phi \circ \Psi) \subset M \times M \setminus B' \times B' $,
  \item $d_{C^0}(\Psi,\id_{M\times M})<\Delta$ and  $d_{C^0}(\Phi \circ \Psi,\id_{M\times M})<\Delta$.
  \end{enumerate}
\end{lemma}

We now explain why this Lemma implies the $C^0$ continuity of $\gamma$ at the identity. Pick any ball $ B \subset M $ such that $ M \setminus B $ has a non-empty interior, and let $ B' \subset B $ be a ball as provided by Lemma \ref{lemma:trick}. Let $ \eps > 0 $, and pick $\Delta$ as provided by Lemma \ref{lemma:variant-sword} for the cases of $ (M \times M, \omega \oplus \omega) $, $ U = B' \times B' $ and $ U = \mathrm{int} (M \times M \setminus B \times B) $.  Moreover, let $\delta$ as also provided by Lemma~\ref{lemma:trick}. Finally let $\phi$ be such that $d_{C^0}(\phi,\id_M)<\delta$ and pick $\Psi$ as given by Lemma \ref{lemma:trick}. By the conclusion (ii) of Lemma \ref{lemma:trick} and by Lemma \ref{lemma:variant-sword} we have $ \gamma(\Psi), \gamma(\Phi \circ \Psi) < \eps $. Therefore, using (\ref{eq:Phi-phi}), the triangle inequality and the duality property, we get 
\begin{align*}
  \gamma(\phi) = \tfrac12\gamma(\Phi) \leq \tfrac12\gamma(\Phi\circ\Psi)+\tfrac12\gamma(\Psi^{-1}) \leq \tfrac\eps2 + \tfrac\eps2 = \eps.
\end{align*}
This shows the continuity of $\gamma$ at identity.

We now turn to the proof of Lemma \ref{lemma:trick}.

\begin{proof} Let $\eps>0$, and let $B $ be a non empty open ball in $M$.

The following claim asserts the existence of a convenient Hamiltonian diffeomorphism which switches coordinates on a small open set.

\begin{claim}\label{claim:switch-coord} There exists a non empty open ball $B''\subset B$ and a Hamiltonian diffeomorphism $f$ on $M\times M$, such that:
  \begin{itemize}
  \item $f$ is the time-1 map of a Hamiltonian supported in $B\times B$,
  \item for all $(x,y)\in B''\times B''$, we have $f(x,y)=(y,x)$.
  \end{itemize}
\end{claim}

\begin{proof} Using a Darboux chart and shrinking $B$ if needed, we may assume without loss of generality that $B$ is a neighborhood of 0 in $\R^{2n}$. Since the space $\mathrm{Sp}(4n,\R)$ of symplectic matrices of $\R^{4n}\simeq \R^{2n}\times\R^{2n}$ is connected, we can choose a path $(A^t)_{t\in[0,1]}$ of such matrices such that $A^0=\id$ and $A^1$ is the linear map $(x,y)\mapsto (y,x)$. Let $B''$ be a small ball containing 0, such that for all $t\in[0,1]$, the closure of $A^t(B''\times B'')$ is included in $B \times B$. Let $Q_t(x)$ be a generating (quadratic) Hamiltonian for $A^t$ and let $\rho$ be a cut-off function supported in $B\times B$ and taking value 1 on $\bigcup_{t\in[0,1]}A^t(B''\times B'')$. The Hamiltonian $F_t(x)=\rho(x) Q_t(x)$ generates a flow which coincides with $A^t$ on $B''\times B''$. Thus, its time-one map $f=\phi_F^1$ suits our needs.
\end{proof}

For the rest of the proof of Lemma \ref{lemma:trick}, we pick a ball $B''$ and a Hamiltonian diffeomorphism $f$ as provided by Claim \ref{claim:switch-coord}. 
%Note that since $f$ is supported in $B'\times B'$, the energy-capacity inequality implies $\gamma(f)\leq 2 e(B'\times B')\leq \frac\eps2$.
Let $B'$ be a ball whose closure is included in $B''$, let $\Upsilon=\phi\times\id_M$ and let \[\Psi=\Upsilon^{-1}\circ f^{-1}\circ \Upsilon\circ f.\] 
%The triangle inequality for $\gamma$ yields $\gamma(\Psi)\leq 2\gamma(f)\leq \eps$, hence Property (i) in Lemma \ref{lemma:trick}. 
If $\phi$ tends to $\id_M$, then $ \Phi $ and $\Psi$ converge to $\id_{M\times M}$, which shows property (ii).

Now, assume that $\phi$ is close enough to $\id_M$ so that $ \Upsilon^{-1}(\supp f) \subset B \times B $. 
Then we have $ \supp \Upsilon^{-1} \circ f^{-1} \circ \Upsilon \subset B \times B $, and since we moreover have 
$ \supp f \subset B \times B $, we conclude that $ \supp \Psi \subset B \times B $. Assume now that $ \phi $ is close enough to $\id_M$ so that in addition we have $\phi(B')\subset B''$. Then for all $(x,y)\in B\times B$, we have
\begin{align*}
  \Phi\circ\Psi(x,y)&= \Phi\circ\Upsilon^{-1}\circ f^{-1}\circ \Upsilon\circ f(x,y)\\
                    &= \Phi\circ\Upsilon^{-1}\circ f^{-1}\circ \Upsilon(y,x)\\
                    &= \Phi\circ\Upsilon^{-1}\circ f^{-1}(\phi(y),x)\\
                    &= \Phi\circ\Upsilon^{-1}(x,\phi(y))\\
  &= \Phi(\phi^{-1}(x),\phi(y))=(x,y).
\end{align*}
Thus, $\Phi\circ\Psi$ coincides with the identity on $B\times B$. This establishes Property (i).
\end{proof}

\begin{remark} \label{rem:Lipschitz}
A slight modification of the proofs of Lemmas \ref{lemma:variant-sword} and \ref{lemma:trick} can show that in fact, $ \gamma $ is locally Lipschitz continuous with respect to the $ C^0 $ metric on $ \Ham(M,\omega) $, that is, we have $ \gamma(\phi) \leqslant C d_{C^0}(\phi,\id_M)$ for every $ \phi $ which is $C^0$-close enough to the identity, where $ C = C(M,\omega) $. 

Indeed, in Lemma  \ref{lemma:variant-sword}, by picking a Morse function $ H : M \rightarrow \R $ with $ \max H - \min H < 1 $, whose critical points all lie in $ U $, we can always take the function $ F $ in the proof of the lemma to be of the form $ F = \eps H $. This shows that in the lemma, for small enough $ \eps $ we can take $ \Delta = c\eps $ where $ c = c(M,\omega) $. Secondly, it is easy to see that also in Lemma \ref{lemma:trick}, for small enough $ \Delta $ we can take $ \delta = C \Delta $, where $ C = C(M,\omega,B) $. Indeed, in the proof of the lemma we have defined $ \Psi=\Upsilon^{-1}\circ f^{-1}\circ \Upsilon\circ f $ which means that $ d_{C^0}(\Psi,\id_M) \leqslant d_{C^0}(\Upsilon^{-1},\id_M) +d_{C^0}(f^{-1}\circ \Upsilon\circ f,\id_M) \leqslant C d_{C^0}(\phi,\id_M) $ where $ C $ depends only on the Lipschitz constants of the diffeomorphism $ f $ and of its inverse.

Indeed, in Lemma  \ref{lemma:variant-sword}, by picking a Morse function $ H : M \rightarrow \R $ with $ \max H - \min H < 1 $, all of whose critical points lie in $ U $, we can always take the function $ F $ in the proof of the lemma to be of the form $ F = \eps H $. This shows that in the lemma, for small enough $ \eps $ we can take $ \Delta = c\eps $ where $ c = c(M,\omega) $. Secondly, it is easy to see that also in Lemma \ref{lemma:trick}, for small enough $ \Delta $ we can take $ \delta = C \Delta $, where $ C = C(M,\omega,B) $. Indeed, in the proof of the lemma we have defined $ \Psi=\Upsilon^{-1}\circ f^{-1}\circ \Upsilon\circ f $ which means that $ d_{C^0}(\Psi,\id_M) \leqslant d_{C^0}(\Upsilon^{-1},\id_M) +d_{C^0}(f^{-1}\circ \Upsilon\circ f,\id_M) \leqslant C d_{C^0}(\phi,\id_M) $ where $ C $ depends only on the Lipschitz constants of the diffeomorphism $ f $ and of its inverse.
\end{remark}

\medskip
We have proved that $\gamma$ is continuous at the identity with respect to the $C^0$-norm. To achieve the proof of Theorem \ref{theo:gamma}, we need the inequality provided by the next lemma.

\begin{lemma}\label{lemma:continuity_gamma} For all homology classes $a,b\neq 0$ and all Hamiltonian diffeomorphisms $\phi,\psi$, we have:
   \begin{equation*}
| \gamma(a,b;\phi) - \gamma(a,b;\psi) | \leq \gamma(\psi^{-1} \circ \phi).
\end{equation*}
\end{lemma}
\begin{proof}
Let $H$ and $F$ be Hamiltonians the time-1 maps of whose flows are $\phi$ and $\psi$, respectively. Then, for any classes $ a , b \in H_*(M) $, we have by the triangle inequality: 
\begin{equation*}
c(a,H) = c(a \cap [M], F \#  (\bar{F} \# H)) \leq c(a,F) + c ([M], \bar{F} \# H).
\end{equation*}
Similarly, $c(a,F)  \leq c(a,H) + c ([M], \bar{H} \# F)$.
Now, by Proposition \ref{prop:sepc_inv}, we have $ c ([M], \bar{H} \# F) = - c([pt],\bar{F} \# H) $.  Thus, we conclude
\begin{equation} \label{eq:theorem-gamma-1}
 c(a,F) + c ([pt],\bar{F} \# H ) \leqslant  c(a, H) \leqslant c(a,F) + c ([M], \bar{F} \# H ).
\end{equation}
Hence,  we also have 
\begin{equation} \label{eq:theorem-gamma-2}
 c(b,F) + c ([pt], \bar{F} \# H ) \leqslant  c(b,H) \leqslant c(b,F) + c ([M], \bar{F} \# H).
\end{equation}
Subtracting \eqref{eq:theorem-gamma-2} from  \eqref{eq:theorem-gamma-1} we obtain 
\begin{align*} 
(c(a,F) &- c(b,F)) - \left( c ([M], \bar{F} \# H ) - c ([pt], \bar{F} \# H) \right) \\
&\leqslant  c(a, H) - c(b,H)\\ &\leqslant \left( c(a,F) - c(b,F) \right) + \left( c ([M], \bar{F} \# H) - c ([pt], \bar{F} \# H) \right),
\end{align*}
which simplifies to
\begin{equation*}
| \gamma(a,b;\phi) - \gamma(a,b;\psi) | \leqslant \gamma(\psi^{-1} \circ \phi).
\end{equation*}
\end{proof}

Lemma \ref{lemma:continuity_gamma} implies that the function $\gamma(a,b; \cdot)$ is continuous at every element $\phi\in \Ham(M,\omega)$ and extends to Hamiltonian homeomorphisms by continuity.

To prove this last fact, let $\phi_i\in \Ham(M,\omega)$ be a sequence which $C^0$-converges to a homeomorphism $\phi$. Then, $\phi_i$ is a Cauchy sequence for the $C^0$-distance. Thus, for all $\eps$, there exists a positive integer $N$ such that for all $i,j>N$, $\phi_i^{-1}\circ\phi_j$ is $C^0$ close enough to $\id_M$ that the above implies $\gamma(\phi_i^{-1}\circ\phi_j)<\eps$. This means that $\phi_i$ is a Cauchy sequence for $\gamma$. Lemma \ref{lemma:continuity_gamma} then gives $|\gamma(a,b;\phi_j)-\gamma(a,b;\phi_i)|<\eps$. Therefore $\gamma(a,b;\phi_i)$ is a Cauchy sequence in $\R$, hence converges. Moreover, if $\phi_i'$ is another sequence which $C^0$ converges to $\phi$, then  $\phi_i^{-1}\circ\phi_i'$ converges to $\id_M$ for the $C^0$ distance, and Lemma \ref{lemma:continuity_gamma} again shows that the limits $\gamma(a,b;\phi_i)$ and $\gamma(a,b;\phi_i')$ are the same. 

The observations of the last paragraph allow to define $\gamma(a,b;\phi)$ for a Hamiltonian homeomorphism $\phi$ as the limit of $\gamma(a,b;\phi_i)$ for any sequence $\phi_i$ which $C^0$-converges to $\phi$. They also imply that the so-defined map $\gamma(a,b;\cdot)$ is continuous on $\overline{Ham}(M,\omega)$ for the $C^0$ topology. 

\begin{remark}\label{remark:proof-non-asperical} We briefly explain now why the above proof can be adapted to general (non necessarily aspherical) closed connected symplectic manifolds, to prove the statement in Remark \ref{remark:non-aspherical}.

  The only part of the proof where symplectic asphericity is required is Lemma \ref{lemma:variant-sword}. However a variant of it was proved in \cite{Sey13} and holds on any closed symplectic manifold:

  \medskip
  \noindent\emph{Let $(M,\omega)$ be a closed symplectically aspherical manifold, let $U$ be a connected  open subset in $M$ and let $(\phi^t)_{t\in[0,1]}\in\mathcal{P}\Ham(M,\omega)$ satisfying:
  \begin{equation*}
  \forall x\in U,\forall t\in[0,1]\ \phi^t(x)=x.
  \end{equation*}
  Then, for every $\eps>0$, there exists $\delta>0$ such that if $d_{C^0}(\phi^t,\id)< \delta$ for all $t\in[0,1]$, then $\gamma((\phi^t)_{t\in[0,1]})< \eps$.}
  \medskip

  The rest of the proof goes through by decorating all our maps with superscripts $t$. More precisely, given a Hamiltonian isotopy $(\phi^t)$, introduce $\Phi^t=\phi^t\times(\phi^t)^{-1}$. The proof then applies almost verbatim. For instance Lemma \ref{lemma:trick} can be adapted so that under the assumption that $d_{C^0}(\phi^t,\id_M)<\delta$ for all $t$, we get a Hamiltonian isotopy $(\Psi^t)$ satisfying:
    \begin{enumerate}[(i)]
  \item $ \supp \Psi^t \subset B \times B $ and $ \supp \Phi^t \circ \Psi^t \subset M \times M \setminus B' \times B' $ for all $ t \in [0,1] $,
  \item $d_{C^0}(\Psi^t,\id_{M\times M}), d_{C^0}(\Phi^t \circ \Psi^t,\id_{M\times M})<\Delta$ for all $ t \in [0,1] $.
  \end{enumerate}
\end{remark}

\section{Proof of Theorem \ref{theo:main}: The generalized Arnold conjecture}\label{sec:Arnold_conj_proof}
We will in fact show the following result, which immediately implies Theorem \ref{theo:main}. Our proof  is a generalization of the one presented in the smooth case in \cite{Howard2012}.

\begin{theo}\label{theo:rigid-ham}
Let $(M, \omega)$ denote a closed, connected and symplectically aspherical  manifold, and let $\phi \in \overline{Ham}(M,\omega)$. If there exist $\alpha,\beta \in  H_*(M) \setminus \{0\}$ with $\mathrm{deg}(\beta)< \dim(M)$, such that $\gamma(\alpha,\alpha \cap \beta; \phi) = 0$, then the set of fixed points of $\phi$ is homologically non-trivial.  
\end{theo}

To prove the above theorem we will need to recall certain aspects of  Lusternik--Schnirelmann theory, which will be done in the next section.   

\subsection{Preparation for the proof: min-max critical values}\label{sec:LS_theory}
Let $M$ be a closed and connected smooth manifold.  Denote by $f \in C^{\infty}(M)$ a smooth function on $M$ and for any $a \in \R$, let $M^{a} = \{x \in M: f(x) < a \}$.   Recall that the inclusion $i_a : M^a \hookrightarrow M$ induces a map $i_a^*: H_*(M^a) \rightarrow H_*(M)$.  Let 
$\alpha \in H_*(M)$ be a non-zero singular homology class and define 
$$\cLS(\alpha,f) := \inf \{a \in \R: \alpha \in \mathrm{Im} (i_a^*) \}.  $$
Note that the numbers $\cLS(\alpha,f)$ already appeared in \ref{prop:sepc_inv}. They are critical values of $f$ and such critical values are often referred to as \emph{homologically essential} critical values. The function $\cLS  : H_*(M) \setminus \{0\} \times C^{\infty}(M) \rightarrow \R$ is often called a \emph{min-max} critical value selector.
In the following proposition $[M]$ denotes the fundamental class of $M$ and $[pt]$ denotes the class of a point.

\begin{prop}\label{prop:cLS}
The min-max critical value selector $\cLS$ possesses the following properties:
\begin{enumerate}

\item $\cLS(\alpha, f)$ is a critical value of $f$,

\item $\min(f) = \cLS([pt],f) \leq \cLS(\alpha,f) \leq \cLS([M],f) = \max(f)$, 

\item $\cLS(\alpha\cap\beta, f) \leq \cLS(\alpha, f)$, for any $\beta \in H_*(M)$ such that $\alpha\cap\beta\neq 0$,

\item Suppose that $\deg(\beta)<\dim(M)$ and  $\cLS( \alpha \cap \beta, f) = \cLS (\alpha, f)$. Then, the set of critical points of $f$ with critical value $\cLS (\alpha, f)$ is homologically non-trivial.
\end{enumerate}
\end{prop}

The above are well-known results from Lusternik-Schnirelmann theory and hence we will not present a proof here. For details, we refer the reader to \cite{LS, cornea-lupton-oprea, viterbo}.

However, for the reader's convenience, we briefly sketch below the proof of the fourth property, in the case $\alpha=[M]$ (which is the only case that we will be using).

\begin{proof}[Proof of Prop \ref{prop:cLS}, point 4. (in the case where $\alpha$ is the fundamental class)]

Since $\cLS([M],f) = \max(f)$, we want to prove that, if $\cLS( \beta, f) = \max(f)$, then the set of points where $f$ reaches its maximum is homologically non-trivial.
Let $\sigma$ be a cycle which represents $\beta$. By definition, the maximum of $f$ on the support of $\sigma$ is at least $c_{LS}(\beta,f)$, which is nothing but $\max(f)$. Thus, $f$ attains its maximum on $\sigma$. We deduce that there is no cycle representing $\beta$ and supported in $M\setminus f^{-1}(\max(f))$. For every neighborhood $U$ of $f^{-1}(\max(f))$, the homology $H_*(M)$ is generated by the homologies  $H_*(U)$ and  $H_*(M\setminus f^{-1}(\max(f)))$. Since $\beta$ cannot be represented in  $M\setminus f^{-1}(\max(f))$, this implies that $U$ has non trivial homology.
\end{proof}

\subsection{The proof}
In this section, we provide the proof of Theorem \ref{theo:rigid-ham}, which immediately implies Theorem \ref{theo:main}.
\begin{proof}[Proof of Theorem \ref{theo:rigid-ham}]
Let $U$ be any  open neighborhood of the fixed-point set of $\phi$. We will show that $\bar{U}$, the closure of $U$, is homologically non-trivial. This clearly implies the theorem.

Let $\phi_i$ be any  sequence of Hamiltonian diffeomorphisms $C^0$-converging to $\phi$, and  for each $i$, we pick a Hamiltonian $H_i$ whose time-one map is $\phi_i$. 
Theorem \ref{theo:gamma} implies that $\gamma(\alpha\cap\beta,\alpha;\phi_i)$ converges to 0 as $i$ goes to $\infty$. Denote by $f : M \rightarrow \R$ a smooth function such that $f = 0$ on $\bar{U}$ and $f < 0$ on $M \setminus \bar{U}$. 

\begin{claim}\label{cl:claim1}
For any $a \in H_*(M)\setminus \{0\}$, there exists $\eps_0>0$ and an integer $i_0$ such that for any $0<\eps\leq \eps_0$ and any $i\geq i_0$, $$c(a, H_i \# \eps f) = c(a, H_i).$$   
\end{claim}
\begin{proof} 
 Let $\delta>0$ be such that $d(\phi(x),x)>\delta$ for all $x\notin U$. The map $\phi$ is the $C^0$-limit of the sequence $\phi_i=\phi_{H_i}^1$, hence there exists some large integer $i_0$ such that:
$$d(\phi_{H_i}^1(x),x)>\tfrac\delta2, \quad\text{for all }x\notin U \text{  and } i \geqslant i_0 .$$
Now take $\eps>0$ so small that $d_{C^0}(\phi_f^{s\eps},\id)<\frac\delta 2$, for all $s\in[0,1]$. If $x$ does not belong to $U$, neither does  $\phi_f^{s\eps}(x)$. Thus, for all $x\notin U$, 
\begin{align*}d(\phi_{H_i}^1\circ\phi_f^{s\eps}(x),x)\geq d(\phi_{H_i}^1\circ\phi_f^{s\eps}(x),\phi_f^{s\eps}(x))-d(\phi_f^{s\eps}(x),x)>\tfrac\delta2-\tfrac\delta2=0.
\end{align*}
In words, $\phi_{H_i}^1\circ\phi_f^{s\eps}$ has no fixed point in $M\setminus U$. Since $f=0$ on $U$, we deduce that the fixed points of $\phi_{H_i}^1\circ\phi_f^{s\eps}$ are the same as those of $\phi_{H_i}^1$. 

Moreover, the actions of the corresponding orbits coincide. Indeed, to see this fact, note that $\phi_{H_i}^1\circ\phi_f^{s\eps}$ can also be generated by the ``concatenated'' Hamiltonian: 
\begin{equation}\label{eq:Ham-concat}
K_{i,\eps}(t,x)=\begin{cases}\rho'(t)s\eps f(x)\quad &\text{ if }t\in[0,\tfrac12]\\
\rho'(t-\tfrac12)H_i(\rho(t),x)\quad&\text{ if }t\in[\tfrac12,1],
\end{cases}
\end{equation}
where $\rho:[0,\frac12]\to[0,1]$ is any smooth non-decreasing function which is 0 near 0 and 1 near $\tfrac12$. It is a standard fact that the paths in $\Ham(M,\omega)$ generated by $H_i \# s\eps f$ and $K_{i,\eps}$ are homotopic with fixed end-points.  Since the mean values of these two Hamiltonians are the same, this implies\footnote{By using the well-known fact \cite{schwarz} that on a closed symplectically aspherical manifold, the action spectrum of a contractible normalized Hamiltonian loop is $ \{ 0 \} $. } that, given a fixed point $x$ of $\phi_{H_i}^1\circ\phi_f^{s\eps}$, the action of the associated 1-periodic orbits will be the same for $H_i \# s\eps f$ and $K_{i,\eps}$.  
Now since $x$ does not belong to the support of $s\eps f$, we easily deduce from \eqref{eq:Ham-concat} that this action is exactly that of $H_i$.

It follows that the spectrum of $H_i\# s\eps f$ (which generates  $\phi_{H_i}^1\circ\phi_f^{s\eps}$) remains constant for $s\in[0,1]$. 
Now continuity of spectral invariants and the fact that the spectrum has measure zero, imply that the number $c(a, H_i \# s \eps f)$ remains constant for $s\in[0,1]$. This proves the Claim.
\end{proof}

It follows from the above claim that for $i$ large enough and $\eps$ small enough, $c(\alpha \cap \beta, H_i \# \eps f) = c(\alpha \cap \beta, H_i)$.
On the other hand, the triangle inequality of Proposition \ref{prop:sepc_inv} implies that $c(\alpha \cap \beta, H_i \# \eps f) \leq c(\alpha , H_i) + c(\beta, \eps f)$.  Thus, for all $i$,  $\gamma(\alpha\cap\beta,\alpha;\phi_i) \leq c(\beta, \eps f)$. Taking limit $i\to\infty$, we  obtain $c(\beta, \eps f)\geq 0$. 

We can now conclude our proof as follows. On one hand Proposition \ref{prop:sepc_inv}.5 implies that for sufficiently small $\eps>0$, one has $$ c(\beta,\eps f)=c_{LS}(\beta, \eps f)=c_{LS}([M]\cap\beta, \eps f).$$ On the other hand, Proposition \ref{prop:cLS}.2  implies $$c_{LS}(\beta,\eps f) \leqslant c_{LS}([M],\eps f)=0.$$ Recalling that $c(\beta, \eps f)\geq 0$, we conclude $$ 0 \leqslant c(\beta,\eps f) = c_{LS}(\beta,\eps f) = c_{LS}([M]\cap\beta, \eps f) \leqslant c_{LS}([M],\eps f) = 0 ,$$ and in particular we obtain the equality $ c_{LS}([M]\cap\beta, \eps f) = c_{LS}([M],\eps f) $. By Proposition \ref{prop:cLS}.4 it follows that the zero level set of $f$, that is $\bar{U}$, is homologically non-trivial. This concludes the proof of Theorem \ref{theo:rigid-ham}. 
\end{proof}

\appendix

\section{Fixed and periodic points of $C^0$ generic Hamiltonian homeomorphisms}

Estimating the number of fixed points of generic (in a $C^1$ sense) Hamiltonian diffeomorphisms has been a central problem in symplectic topology over the past 30 years. The construction of Floer homology implies that this number is bounded below by the sum of the Betti numbers of the manifold. It is therefore natural to ask if similar estimates hold for $C^0$-generic Hamiltonian homeomorphisms. It turns out that the situation is dramatically different, as the next proposition shows. 

\begin{prop}\label{prop:generic}
 Let $(M, \omega)$ be any closed symplectic manifold. There exists a residual\footnote{By residual subset we mean a countable intersection of dense open subsets.} subset $\mathcal{U}$ of $\overline{\Ham}(M,\omega)$ such that every element in $\mathcal{U}$ has infinitely many fixed points.
\end{prop}

In the case of a symplectic surface $(\Sigma, \omega)$, a stronger result holds: generically in $\overline{\Ham}(\Sigma,\omega)$, the set of fixed points is a Cantor set; see \cite{guiheneuf}.
The proof uses tools that are not available in higher dimension, therefore it is not clear to us if this result extends to higher dimension or not.  

The proof of Proposition \ref{prop:generic} will follow easily from the Lefschetz index theory. The Lefschetz index is an integer associated to an isolated fixed point of a continuous map. Of its properties, we will use the following  (see e.g. \cite{Katok-H}, Chapter 2, Section 8.4):
  \begin{itemize}
  \item[(a)] The index of a non-degenerate fixed point of a diffeomorphism is either $1$ or $-1$.
  \item[(b)] Fixed points with non-zero index are $C^0$-stable. More pecisely, if $x_0$ is an isolated fixed point of a continuous map $f$, then there exists a $C^0$-neighborhood of $f$ such that every map in this neighborhood admits a fixed point near $x_0$.
  \end{itemize}
We are now ready for the proof. 

\begin{proof}
Let $\phi$ be a non-degenerate Hamiltonian diffeomorphism which admits at least $N$ fixed points. According to Property (a) above, all its fixed points have non-zero index, hence by Property (b), there exists a $C^0$-open subset $U_\phi\subset \overline{\Ham}(M,\omega)$ such that every element in $U_\phi$ has at least $N$ fixed points.
 
Now, it is well known that non-degenerate Hamiltonian diffeomorphisms are dense in $\Ham(M,\omega)$, hence in $\overline{\Ham}(M,\omega)$. Furthermore, given a non-degenerate Hamiltonian diffeomorphism $\psi$, we can always perform a local modification near one of its fixed points (which are known to exist by Floer homology theory), to construct a new Hamiltonian diffeomorphism, $C^0$-close to $\psi$ and admitting at least $N$ fixed points. In other words, the set $V_N$ of non-degenerate Hamiltonian diffeomorphisms admitting at least $N$ fixed points is dense in $\overline{\Ham}(M,\omega)$.  

It follows immediately that the set $\mathcal{U}_N=\bigcup_{\phi\in V_N}U_\phi$ is a dense open subset and that all its elements have at least $N$ fixed points. Finally, $\mathcal{U}=\bigcap_{N>0}U_N$ is a residual subset all of whose elements have infinitely many fixed points.
\end{proof}

We end the paper with a brief discussion of the Conley conjecture on closed and symplectically aspherical manifolds.  The conjecture, which was proven in \cite{hingston, ginzburg}, states that a Hamiltonian diffeomorphism must have infinitely many periodic points.  When it comes to Hamiltonian homeomorphisms the conjecture was proven on surfaces in \cite{lecalvez06}.  In dimensions four and higher, we are neither able to prove it nor are we able to adapt our counterexample from \cite{BHS} to disprove it.  However, we believe that one should be able to adapt the proof of Proposition \ref{prop:generic} to show that, for a $C^0$--generic Hamiltonian homeomorphism $\phi$, the set $\mathcal{P}_k(\phi)$ consisting of periodic points of period $k$, must be infinite for infinitely many values of $k$.

% #############################################
%
%   BIBLIO
%
% #############################################

\bibliographystyle{abbrv}
\bibliography{biblio}

%%%%%%%%%%%%%%%%%%%%%%%%%%%%%%%%%%%%%%%%%%%%%%%%%%%%%%%%%%%%%%%%%%%%%%%%%%%%%%%%
%%%%%%%%%%%%%%%%%%%%%%%%%%%%%%%%%%%%%%%%%%%%%%%%%%%%%%%%%%%%%%%%%%%%%%%%%
%%%%%%%%%%%%%%%%%%%%%%%%%%%%%%%%%%%%%%%%%%%%%

{\small

\medskip
\noindent Lev Buhovski\\
School of Mathematical Sciences, Tel Aviv University \\
{\it e-mail}: levbuh@post.tau.ac.il
\medskip

\medskip
\noindent Vincent Humili\`ere \\
\noindent Sorbonne Universit\'e, Universit\'e de Paris, CNRS, Institut de Math\'ematiques de Jussieu-Paris Rive Gauche, F-75005 Paris, France.\\
{\it e-mail:} vincent.humiliere@imj-prg.fr
\medskip

\medskip
 \noindent Sobhan Seyfaddini\\
\noindent Sorbonne Universit\'e, Universit\'e de Paris, CNRS, Institut de Math\'ematiques de Jussieu-Paris Rive Gauche, F-75005 Paris, France.\\
{\it e-mail:} sobhan.seyfaddini@imj-prg.fr}

\end{document}